\IfFileExists{iopjournal.cls}{\documentclass{iopjournal}}{\documentclass{iopjournal-draft}}

\usepackage[T1]{fontenc}
\usepackage{graphicx}
\usepackage{xcolor}
\usepackage{hyperref}
\usepackage[utf8]{inputenc}
\usepackage{lmodern}
\usepackage{amsmath,amssymb,amsthm,mathtools}
\usepackage{bm}
\usepackage{microtype}
\usepackage[numbers,sort&compress]{natbib}
\usepackage{xurl}
\usepackage{booktabs}
\usepackage{enumitem}
\usepackage{listings}
\usepackage{placeins}
\usepackage{tikz}

\newtheorem{theorem}{Theorem}[section]
\newtheorem{proposition}[theorem]{Proposition}
\newtheorem{lemma}[theorem]{Lemma}
\newtheorem{corollary}[theorem]{Corollary}
\theoremstyle{definition}

\theoremstyle{remark}
\newtheorem{remark}[theorem]{Remark}

\newcommand{\D}{\mathbb D}
\newcommand{\dd}{\,\mathrm d}
\newcommand{\sech}{\operatorname{sech}}
\newcommand{\csch}{\operatorname{csch}}
\newcommand{\Dom}{\operatorname{Dom}}

\newcommand{\spec}{\operatorname{spec}}
\newcommand{\norm}[1]{\lVert #1\rVert}
\newcommand{\ip}[2]{\langle #1,#2\rangle}

\newcommand{\R}{\mathbb R}
\newcommand{\N}{\mathbb N}
\newcommand{\Z}{\mathbb Z}

\numberwithin{equation}{section}

\hypersetup{
  pdftitle={Explicit diagonalization of the linearized Calder\'on problem on the disk: multiplicity normalization and exponential spectral attenuation},
  pdfauthor={William R. B. Lionheart},
  colorlinks=true,
  linkcolor=blue!45!black,
  citecolor=blue!45!black,
  urlcolor=blue!45!black
}

\begin{document}

\articletype{Paper}

\title{Explicit diagonalization of the linearized Calder\'on problem on the disk: multiplicity normalization and exponential spectral attenuation}

\author{William R B Lionheart$^{1,*}$}

\affil{$^1$Department of Mathematics, The University of Manchester, Manchester M13 9PL, United Kingdom}
\affil{$^*$Author to whom correspondence should be addressed.}
\email{bill.lionheart@manchester.ac.uk}


\begin{abstract}
We study the full angular dependence of the linearized Schr\"odinger form of the Calder\'on problem on the unit disk.  We obtain an explicit continuous spectral diagonalization of its normal operator and hence a generalized singular system for the linearized forward map.  To our knowledge, this is the first such diagonalization for the full angularly dependent, multiplicity-normalized linearized measurement map on the disk.

An important step is to collapse the exact redundancy classes in the boundary Fourier matrix and to assign equal weight to each algebraically independent moment.  The resulting data are indexed by the Zernike lattice, and their normal operator is
\[
 (Kf)(z)=\frac1\pi\int_{\D}\frac{f(w)}{|1-z\overline w|^2}\,\dd A(w).
\]
A weighted Schur estimate gives the exact norm $\|K\|=\pi$, and a weakly null sequence of boundary annuli proves non-compactness before spectral theory is invoked.  We identify a second-order differential expression on the disk satisfying a direct global kernel commutator identity.  Multiplication by $(1-|z|^2)/2$ conjugates its canonical realization to the Poincar\'e Laplacian and conjugates $K$ to hyperbolic convolution with $(4\pi)^{-1}\sech^2(d_H/2)$.  This is the renormalized $s=1$ endpoint of the weighted disk Berezin family.  Classical hyperbolic harmonic analysis then yields
\[
 K=\pi\sech\!\left(\pi\sqrt{-\mathsf L-\tfrac14}\right),
\]
where $\mathsf L$ is the Euclidean-space realization unitarily equivalent to the hyperbolic Laplacian.  Hence $K$ has purely absolutely continuous spectrum $[0,\pi]$, and its generalized singular attenuation is exponential in hyperbolic frequency.  We give the associated Legendre and hypergeometric modes and a brief finite-section numerical check.  The abstract Berezin diagonalization belongs to established harmonic analysis; the new inverse-problem result is the derivation of this endpoint operator from the multiplicity-normalized linearized Calder\'on measurements.
\end{abstract}

\section{Introduction}

Calder\'on introduced the inverse boundary problem that now bears his name in the conductivity setting \cite{Calderon1980}, including calculation of the Fr\'echet derivative and a proof of its injectivity.  This is  important for medical, industrial and geophysical imaging, where the technique goes by various names including Electrical Impedance Tomography (EIT) and Electrical Resistivity Tomography  \cite{adler2015electrical}.

The singular value decomposition (SVD) of discretized EIT Jacobians has been used to study ill-posedness and non-uniform resolution since the mid-1980s.  Murai and Kagawa \cite{murai1985electrical} computed a numerical SVD for a resistor-mesh/finite-element model.  Breckon and Pidcock \cite{breckon1988some} used singular values from a finite-element model to assess the relative importance of drive-pair separation, and Connolly and Wall \cite{connolly1988inverse} published numerical singular values for a rotationally symmetric problem.  Numerical singular functions for the disk were presented by the author at the Third European Workshop on EIT in Copenhagen in 1990 and in the author's doctoral thesis, under his previous name  W. R. Breckon \cite{breckon1990image}, see also \cite{adler2021,Lionheart2026AtLast}.

These computations already showed severe ill-posedness, strongly non-uniform resolution, and weak sensitivity to interior perturbations.  Nevertheless, an explicit analytic singular system for the full angular problem has remained unavailable for roughly four decades.  The contrast with the classical Zernike SVD of the Radon transform on the disk made this absence particularly conspicuous.  A key change of viewpoint came from recognizing that the continuum normal operator is not compact (Proposition~\ref{prop:noncompact}); the appropriate analogue of an SVD is therefore a continuous spectral diagonalization rather than a discrete singular sequence.

A parallel line of work sought to explain the geometry underlying the
Barber--Brown backprojection method \cite{barber1986recent}.  Santosa and Vogelius interpreted
that construction in terms of a generalized Radon transform, while
Berenstein and Casadio Tarabusi subsequently gave an exact
factorization of the linearized conductivity problem through the
geodesic X-ray transform on the Poincar\'e disk and an additional
hyperbolic convolution
\cite{SantosaVogelius1990,
	BerensteinCasadio1994,
	BerensteinCasadio1996}.
Thus hyperbolic geometry was already known to be intrinsic to the disk
problem, although the linearized EIT map is not itself a hyperbolic
X-ray transform.  The additional convolution remained a separate
inversion problem, and this approach did not provide an explicit
generalized singular system for the complete boundary measurement map.
The role of this geometry in the
normal operator considered here is made precise below.

 Recent work of Garde, Hyv\"onen and collaborators \cite{garde2024linearized, autio2025linearization}    represents the linearized forward map in a Fourier–Zernike basis, in which it decomposes into triangular systems indexed by angular frequency. Our approach is complementary: rather than triangularizing the forward map in a polynomial basis, we diagonalize the rotationally symmetric normal operator using a continuous family of generalized hypergeometric eigenfunctions.

The radial subproblem---that is, the $m=0$ angular Fourier block corresponding to rotationally invariant perturbations---is particularly transparent.  For rotationally symmetric perturbations, a multiplicity-normalized normal operator reduces to the Hilbert kernel \cite{Lionheart2026AtLast}
\begin{equation}
 (Hf)(t)=\int_0^1\frac{f(s)}{1-st}\,\dd s.
 \label{eq:radial-hilbert}
\end{equation}
Its continuous spectral multiplier may be obtained from an adjacent-interval finite Hilbert transform, from a commuting Sturm--Liouville expression, from Hilbert-matrix theory, or from a P\"oschl--Teller scattering problem; see, for example, \cite{KatsevichTovbis2016,AlemanMontesSarafoleanu2012,MontesVirtanen2025}.  The finite-Hilbert-transform route is especially informative about non-compactness: the intervals share an endpoint.  Introducing a gap---which in the present problem corresponds to knowing the perturbation in a neighbourhood of the boundary---makes the radial operator compact.  Compactness is also restored for the frequency-weighted Neumann-to-Dirichlet formulation, but an explicit singular system is not presently known in either variant.

The purpose of this paper is to solve the full angular problem for a symmetry-adapted, equal-weight non-redundant Hilbert realization of the derivative at the zero potential.  Rather than a  collection of separated radial formulas,  all angular blocks assemble into one hyperbolically invariant operator on the disk.  To our knowledge, this gives the first explicit diagonalization of the full angularly dependent, multiplicity-normalized linearized Calder\'on measurement map in this setting and resolves an analytic question that has been investigated numerically since the earliest EIT SVD studies.

The key bookkeeping observation is that a matrix element indexed by boundary Fourier frequencies $(k,l)$ depends on the two integers
\begin{equation}
 m=k-l,\qquad n=|k|+|l|.
 \label{eq:intro-invariants}
\end{equation}
The admissible pairs satisfy
\[
 n\geq |m|,\qquad n\equiv m\pmod 2,
\]
which is exactly the Zernike index lattice. Compare for example the diagonalization of the longitudinal tensor ray transform on the disk using Zernike disk functions \cite{KazantsevBukhgeim2004}.  Counting every independent datum once, rather than with its raw matrix multiplicity, gives the moment map whose normal kernel is
\begin{equation}
 \kappa(z,w)=\frac1{\pi|1-z\overline w|^2}.
 \label{eq:intro-kernel}
\end{equation}
After a unitary change of measure it is instead the squared modulus of the inner product between normalized Hardy reproducing kernels  and is a limiting Berezin convolution ( see Sec \ref{sec:berezin}).  Berezin transforms and their expression as functions of invariant differential operators have a substantial literature \cite{Berezin1975,UnterbergerUpmeier1994,Englis1995,Stroethoff1997}.  Our contribution is the bridge from linearized Calder\'on measurements to the Hardy-endpoint kernel, together with its inverse-problem interpretation.
The multiplicity normalization is the step that reveals this invariant operator.  
Section~\ref{sec:nonredundant} compares it explicitly with the raw entrywise matrix norm and with physically noise-weighted data.  In particular, counting each independent datum once cannot increase data energy or the ordered singular values relative to counting every repeated copy in raw $\ell^2$; it removes that repetition gain.

Unlike the earlier X-ray factorization, the hyperbolic structure found
here arises directly in the normal operator of the symmetry-reduced
Fourier measurement map.

The main results may be summarized as follows.
\begin{enumerate}[label=(\roman*)]
\item The Schr\"odinger Dirichlet-to-Neumann map is Fr\'echet differentiable at the zero potential for perturbations in $L^2(\D)$, and its derivative is the usual Alessandrini bilinear form.
\item The multiplicity-reduced Fourier data define a bounded moment map $\mathcal{M}$ with 
\[
 K=\mathcal{M}^*\mathcal{M},
 \]
and kernel \eqref{eq:intro-kernel}.
\item The operator is positive and self-adjoint, has exact norm $\pi$, and is non-compact.
\item A global second-order differential expression $\mathcal P$ satisfies
\[
 \mathcal P_z\kappa(z,w)=\mathcal P_w\kappa(z,w),
\]
not merely after separation into angular modes.
\item The hyperbolic geometry enters the present construction at the
level of the normal operator.  Under a unitary multiplication operator
$U$, the auxiliary expression
\[
\mathcal L=\mathcal P-\frac12\partial_\theta^2-1
\]
is transformed into the curvature $-1$ Poincar\'e Laplacian, while
$UKU^{-1}$ is an isometry-invariant radial convolution on the
hyperbolic disk.
\item The exact normal-operator multiplier is
\[
\lambda(\mu)=\frac{\pi}{\cosh(\pi\mu)},
\]
so that the spectrum is continuous and high hyperbolic frequencies
are exponentially attenuated.

\item The generalized eigenfunctions of $K$, and hence the generalized
right singular functions of $\mathcal M$, are given explicitly by
associated Legendre or hypergeometric functions.  Their Liouville
normal form is a P\"oschl--Teller scattering equation, which explains
the concentration of fine radial oscillations near the boundary.

\end{enumerate}

The distinction between multiplicity normalization and physical frequency weighting is important.  The former assigns one unit of weight to each algebraically independent tangent datum.  A Sobolev, covariance-whitening, Dirichlet-to-Neumann, or Neumann-to-Dirichlet weighting may introduce additional factors depending on $k$ and $l$ and therefore changes the normal operator.  The paper treats the equal-weight, non-redundant tangent data space exactly; we do not claim that this is the unique physical noise norm for an electrode-based experiment.

\section{Linearization in an $L^2$ potential}
\label{sec:linearization}

Let $\D\subset\R^2$ be the unit disk and consider
\begin{equation}
 (-\Delta+q)u=0\quad\hbox{in }\D,
 \qquad u|_{\partial\D}=f.
 \label{eq:schrodinger-pde}
\end{equation}
For $q$ close to zero in $L^2(\D)$, write $\Lambda_q:H^{1/2}(\partial\D)\to H^{-1/2}(\partial\D)$ for the Dirichlet-to-Neumann map, with the outward normal convention. 
We work initially in this Schr\"odinger formulation, and hence the terminology describing $q$ as a potential.  Its relation to
the conductivity equation and to the linearized conductivity
Dirichlet-to-Neumann map is recorded after Proposition~\ref{prop:linearization}.

For a positive conductivity $\gamma$, the substitution
$u=\gamma^{1/2}v$ transforms
\begin{equation} \label{eq:conductivity}
\nabla\cdot(\gamma\nabla v)=0
\end{equation}
into
\begin{equation}\label{eq:gammaq}
(-\Delta+q_\gamma)u=0,
\qquad
q_\gamma=\gamma^{-1/2}\Delta\gamma^{1/2}.
\end{equation}
If $\gamma=1$ in a neighbourhood of $\partial\D$, the corresponding
boundary maps agree under this transformation.  For general boundary
values of $\gamma$ there is also the standard boundary conjugation and
zeroth-order correction.

We use the standard Sobolev scale on the boundary circle.  Using the Fourier series
\[
\varphi_k(\theta)=(2\pi)^{-1/2}e^{ik\theta},
\qquad
f=\sum_{k\in\mathbb Z}\widehat f_k\varphi_k,
\]
we set
\[
\norm{f}_{H^s(\partial\D)}^2
=
\sum_{k\in\mathbb Z}(1+k^2)^s|\widehat f_k|^2.
\]
Thus $H^{-1/2}(\partial\D)$ is the dual of
$H^{1/2}(\partial\D)$ with respect to the $L^2$ pairing.
Any equivalent Sobolev norms would give the same operator spaces.
For the interior space we use
\[
\norm{u}_{H^1(\D)}^2
=
\norm{u}_{L^2(\D)}^2+\norm{\nabla u}_{L^2(\D)}^2.
\] 

The following calculation of the Fr\'echet derivative is a retelling of the classical result of Calder\'on \cite{Calderon1980} in the Schr\"odinger context. 
 
\begin{proposition}[Linearization at the zero potential]
\label{prop:linearization}
There is a neighbourhood of zero in $L^2(\D)$ on which \eqref{eq:schrodinger-pde} is uniquely solvable and $q\mapsto\Lambda_q$ is Fr\'echet analytic as a map
\[
 L^2(\D)\longrightarrow
 \mathcal L\bigl(H^{1/2}(\partial\D),H^{-1/2}(\partial\D)\bigr).
\]
In particular,
\begin{equation}
 \ip{D\Lambda_0[\eta]f}{g}
 =\int_{\D}\eta\,u_f^0\,\overline{u_g^0}\,\dd A,
 \label{eq:linearized-DtN}
\end{equation}
where $u_f^0$ and $u_g^0$ are harmonic extensions.  Moreover,
\begin{equation}
 \norm{D\Lambda_0[\eta]}_{H^{1/2}\to H^{-1/2}}
 \leq C\norm{\eta}_{L^2},
 \label{eq:derivative-bound}
\end{equation}
and
\begin{equation}
 \norm{\Lambda_q-\Lambda_0-D\Lambda_0[q]}_{H^{1/2}\to H^{-1/2}}
 \leq C\norm{q}_{L^2}^2
 \label{eq:quadratic-remainder}
\end{equation}
for $\norm q_{L^2}$ sufficiently small.
\end{proposition}

\begin{proof}
For $u\in H^1(\D)$ and $v\in H_0^1(\D)$, H\"older's inequality with exponents $2,4,4$, followed by the Sobolev embeddings $H^1(\D),H_0^1(\D)\hookrightarrow L^4(\D)$, gives
\begin{equation}
 \left|\int_{\D}q\,u\overline v\,\dd A\right|
 \leq \norm q_{L^2}\norm u_{L^4}\norm v_{L^4}
 \leq C\norm q_{L^2}\norm u_{H^1}\norm v_{H_0^1}.
 \label{eq:multiplication-estimate}
\end{equation}
Consequently, multiplication by $q$ defines a bounded map $M_q:H^1(\D)\to H^{-1}(\D)$ and
\[
 \norm{M_qu}_{H^{-1}}\leq C\norm q_{L^2}\norm u_{H^1}.
\]
Let $G=(-\Delta_D)^{-1}:H^{-1}(\D)\to H_0^1(\D)$.  If $u_f^0$ denotes harmonic extension, the solution is
\begin{equation}
 u_f^q=(I+GM_q)^{-1}u_f^0.
 \label{eq:operator-series}
\end{equation}
For small $\norm q_2$, the inverse is represented by a norm-convergent Neumann series, proving analyticity and
\[
 \norm{u_f^q-u_f^0}_{H^1}
 \leq C\norm q_2\norm f_{H^{1/2}}.
\]
The integration-by-parts formula commonly called Alessandrini's identity is
\begin{equation}
 \ip{(\Lambda_q-\Lambda_0)f}{g}
 =\int_{\D}q\,u_f^q\,\overline{u_g^0}\,\dd A.
 \label{eq:alessandrini}
\end{equation}
Subtraction of \eqref{eq:linearized-DtN}, followed by \eqref{eq:multiplication-estimate}, gives \eqref{eq:quadratic-remainder}.  The same estimate with harmonic extensions gives \eqref{eq:derivative-bound}.
\end{proof}
\begin{remark}[Relation with conductivity perturbations]
	For a positive conductivity $\gamma\in H^2(\D)$, the Liouville
	substitution $u=\gamma^{1/2}v$ transforms \eqref{eq:conductivity} into  \eqref{eq:gammaq}. 
	Since $H^2(\D)$ is a Banach algebra and embeds into 
	$C^0(\overline\D)$, the map
	\[
	\gamma\longmapsto q_\gamma
	\]
	is locally real analytic from the open set of strictly positive
	$H^2$ conductivities into $L^2(\D)$.  At the constant conductivity,
	\[
	Dq_1[h]=\frac12\Delta h.
	\]
	
	If $\gamma=1$ in a neighbourhood of $\partial\D$, the conductivity
	and Schr\"odinger Dirichlet-to-Neumann maps agree under the Liouville
	transformation.  Consequently,
	\[
	D\Lambda^{\rm c}_1[h]
	=
	D\Lambda^{\rm s}_0\left[\frac12\Delta h\right],
	\]
	and Proposition~\ref{prop:linearization} gives
	\[
	\norm{D\Lambda^{\rm c}_1[h]}_{H^{1/2}\to H^{-1/2}}
	\leq C\norm h_{H^2}.
	\]
	Without the boundary assumption, the two derivatives differ by
	explicit terms involving the boundary traces of $h$ and
	$\partial_\nu h$.  The diagonalization below is formulated in the
	potential variable $\eta$; in the conductivity variable the forward
	map is precomposed with $\frac12\Delta$.
\end{remark}
\begin{remark}
The operator-series argument gives every Fr\'echet derivative, but only the first derivative is needed below.  
\end{remark}

We now remain in the Schr\"odinger formulation and compute the
linearized measurements in the boundary Fourier basis.
Let
\begin{equation}
 \varphi_k(\theta)=(2\pi)^{-1/2}e^{ik\theta},\qquad k\in\Z.
 \label{eq:boundary-fourier}
\end{equation}
Its harmonic extension is $u_k^0(r,\theta)=r^{|k|}\varphi_k(\theta)$.  If
\[
 \eta_m(r)=\frac1{2\pi}\int_0^{2\pi}\eta(r,\theta)e^{-im\theta}\,\dd\theta,
\]
then \eqref{eq:linearized-DtN} yields
\begin{equation}
 \ip{D\Lambda_0[\eta]\varphi_k}{\varphi_l}
 =\int_0^1\eta_{l-k}(r)r^{|k|+|l|}r\,\dd r.
 \label{eq:fourier-matrix-element}
\end{equation}
Thus, up to the harmless sign convention for the object angular index, the datum depends only on $k-l$ and $|k|+|l|$.

\section{Non-redundant measurements and the Zernike lattice}
\label{sec:nonredundant}

\begin{figure}[htbp]
	\centering
	
	\begin{tikzpicture}[scale=0.68]
		\draw[->] (-5.2,0)--(5.2,0) node[right] {$k$};
		\draw[->] (0,-5.2)--(0,5.2) node[above] {$l$};
		
		\foreach \i in {-4,...,4}
		\foreach \j in {-4,...,4}
		\fill[gray!45] (\i,\j) circle (0.035);
		
		\draw[very thick,blue]
		(4,0)--(0,4)--(-4,0)--(0,-4)--cycle;
		\node[blue] at (-3.6,2.0) {$n=4$};
		
		\draw[thick,red] (-2.2,-4.2)--(5.0,3.0);
		\node[red] at (3.6,2.55) {$m=2$};
		
		\fill[red] (3,1) circle (0.10);
		\fill[red] (-1,-3) circle (0.10);
		
		\draw[very thick,orange] (0,-4)--(4,0);
		\node[orange!80!black] at (3.9,-1.0) {$m=4$};
		
		\foreach \j in {0,...,4}
		\fill[orange!80!black] (\j,\j-4) circle (0.10);
		
		\node at (0,-5.0) {$(k,l)$ coordinates};
	\end{tikzpicture}
	\hspace{1cm}
	\begin{tikzpicture}[scale=0.68]
		\draw[->] (-5.2,0)--(5.2,0) node[right] {$m=k-l$};
		\draw[->] (0,-5.2)--(0,5.2) node[above] {$v=k+l$};
		
		\draw[very thick,blue]
		(-4,-4)--(4,-4)--(4,4)--(-4,4)--cycle;
		\node[blue] at (-2.8,4.45)
		{$n=\max\{|m|,|v|\}=4$};
		
		\draw[thick,red] (2,-4.7)--(2,4.7);
		\fill[red] (2,4) circle (0.10);
		\fill[red] (2,-4) circle (0.10);
		\node[red] at (2.8,2.8) {$m=2$};
		
		\draw[very thick,orange] (4,-4)--(4,4);
		\foreach \v in {-4,-2,0,2,4}
		\fill[orange!80!black] (4,\v) circle (0.10);
		\node[orange!80!black] at (4.75,-0.4) {$m=4$};
		
		\node at (0,-5.0) {$(m,v)$ coordinates};
	\end{tikzpicture}
	
\caption{
	Geometry of the redundancy fibres.  In the $(k,l)$ plane, the level
	set $|k|+|l|=n$ is a diamond, while $k-l=m$ is a line of slope one.
	For $n>|m|$ and matching parity there are two lattice intersections.
	When $n=|m|$, the line coincides with one side of the diamond and the
	fibre contains $|m|+1$ lattice points.  In the coordinates
	$m=k-l$, $v=k+l$, the same statement becomes
	$n=\max\{|m|,|v|\}$; the diamond is transformed into a square.
}
	\label{fig:fibre-geometry}
\end{figure}

Define
\begin{equation}
 m=k-l,\qquad n=|k|+|l|.
 \label{eq:mn}
\end{equation}
For fixed $(m,n)$ let
\[
 F_{m,n}=\{(k,l)\in\Z^2:k-l=m,\ |k|+|l|=n\}.
\]

\begin{proposition}[Fibre classification]
\label{prop:fibres}
The `fibre' $F_{m,n}$is nonempty precisely when
\begin{equation}
 n\geq |m|,\qquad n\equiv m\pmod 2.
 \label{eq:zernike-lattice}
\end{equation}
For such admissible $(m,n)$,
\begin{equation}
 \#F_{m,n}=:\nu(m,n)=
 \begin{cases}
 |m|+1,&n=|m|,\\
 2,&n>|m|.
 \end{cases}
 \label{eq:fibre-multiplicity}
\end{equation}
\end{proposition}

\begin{proof}
	Figure~\ref{fig:fibre-geometry} suggests the classification
	geometrically: generically the line $k-l=m$ meets the diamond
	$|k|+|l|=n$ in two lattice points, while for $n=|m|$ it coincides
	with one side.  The following change of variables turns this picture
	into an algebraic proof.
Put $v=k+l$.  The identity
\begin{equation}
 |k|+|l|=\max\{|k-l|,|k+l|\}
 \label{eq:max-identity}
\end{equation}
gives $n=\max\{|m|,|v|\}$.  Moreover,
\[
 k=\frac{m+v}{2},\qquad l=\frac{v-m}{2},
\]
so $v$ and $m$ have the same parity.  If $n>|m|$, then $v=\pm n$, giving two points.  If $n=|m|$, then $v=-n,-n+2,\ldots,n$, giving $n+1=|m|+1$ points.
\end{proof}

The sum-and-difference coordinates
\begin{equation}
 p=\frac{n+m}{2},\qquad q=\frac{n-m}{2}
 \label{eq:pq}
\end{equation}
give a bijection between the lattice \eqref{eq:zernike-lattice} and $\N_0^2$, with
\[
 n=p+q,\qquad m=p-q,
\]
and radial Zernike index
\begin{equation}
 j=\min\{p,q\}=\frac{n-|m|}{2}.
 \label{eq:radial-index}
\end{equation}
This is the familiar bidegree-to-polar-degree relation
\[
 z^p\overline z^{\,q}=r^{p+q}e^{i(p-q)\theta}.
\]

When a redundant data matrix is constant on each fibre, the norm
\begin{equation}
 \norm Y_\sharp^2=
 \sum_{k,l\in\Z}\frac{|Y_{kl}|^2}{\nu(k-l,|k|+|l|)}
 \label{eq:sharp-data-norm}
\end{equation}
counts every independent datum exactly once.  Equivalently, one may select one representative for each $(m,n)$ or $(p,q)$.

\subsection{The topology of the reduced data}
Let
\[
 \Gamma=\{(m,n)\in\Z\times\N_0:n\geq |m|,\ n\equiv m\pmod 2\},
 \qquad \pi(k,l)=(k-l,|k|+|l|).
\]
A consistent linearized Fourier matrix has the form $Y_{kl}=d_{\pi(k,l)}$ for a reduced sequence $d=(d_\gamma)_{\gamma\in\Gamma}$.  The ordinary entrywise matrix norm restricted to this consistency subspace is
\begin{equation}
 \norm{Y}_{\mathrm{raw}}^2
 =\sum_{\gamma\in\Gamma}\nu(\gamma)|d_\gamma|^2,
 \label{eq:raw-data-norm}
\end{equation}
whereas \eqref{eq:sharp-data-norm} gives
\begin{equation}
 \norm{Y}_{\sharp}^2
 =\sum_{\gamma\in\Gamma}|d_\gamma|^2.
 \label{eq:reduced-data-norm}
\end{equation}
Thus the reduced norm is not simply the topology induced by raw matrix $\ell^2$.  It is the equal-weight Hilbert norm on the algebraically independent tangent measurements.  Since $\nu\geq1$,
\[
 \norm{Y}_{\sharp}\leq \norm{Y}_{\mathrm{raw}}.
\]
On the truncation $|m|\leq M$ the converse estimate
\[
 \norm{Y}_{\mathrm{raw}}\leq\sqrt{M+1}\,\norm{Y}_{\sharp}
\]
holds, but there is no uniform equivalence as $M\to\infty$ because $\nu(m,|m|)=|m|+1$.

Equivalently, if $R$ replicates a reduced sequence across its fibres, then on every finite truncation the raw measurement map is $R\mathcal M$ and
\[
 (R\mathcal M)^*(R\mathcal M)=\mathcal M^*D_\nu\mathcal M\geq\mathcal M^*\mathcal M,
 \qquad D_\nu=\operatorname{diag}(\nu).
\]
By the min--max principle, on every finite truncation the ordered singular values satisfy
\[
 s_j(\mathcal M)\leq s_j(R\mathcal M).
\]
The severe exponential attenuation is already present in the radial block. In that block the zero moment occurs once, whereas every positive moment occurs twice, so the raw and equal-weight reduced norms are uniformly equivalent,
\[||d||_{\mbox{red}} \leq ||d||_{\mbox{raw}} \leq \sqrt{2}\,||d||_{\mbox{red}},\]
although they are not scalar multiples. The corresponding radial normal operators differ, after an overall factor, by the rank-one correction associated with the zero moment. Hence the exponential degree of ill-posedness is not produced by the unbounded multiplicities of the non-radial data.

Note that the norm is also not the induced by the Banach operator norm on $\mathcal L(H^{1/2},H^{-1/2})$.  An SVD requires a Hilbert data metric, and the choice here is attached to the derivative at the rotationally invariant background $q=0$.  We do not assert that the nonlinear map takes values in this sequence space.  For physical electrode data the statistically appropriate metric depends on the noise covariance: independent repeated measurements retain a multiplicity gain, while exact or perfectly correlated algebraic copies should be counted only once.

Define the non-redundant moment map initially on a dense subspace by
\begin{equation}
 \mathcal{M} f_{p,q}
 =\frac1{\sqrt\pi}\int_{\D}
 f(w)\overline w^{\,p}w^q\,\dd A(w),
 \qquad p,q\in\N_0.
 \label{eq:M}
\end{equation}
In $(m,j)$ coordinates this is
\begin{equation}
 (\mathcal{M}  f)_{m,j}
 =\frac1{\sqrt\pi}\int_{\D}
 f(re^{i\theta})r^{|m|+2j}e^{-im\theta}\,\dd A,
 \qquad m\in\Z,\quad j\in\N_0.
 \label{eq:sigma-mj}
\end{equation}
The normalization in \eqref{eq:M} is chosen to make the later kernel and norm formulas simple.

\section{The normal operator}
\label{sec:normal}
In this section we show that the non-redundant data norm has produced a simple form for the kernel of the normal operator.
\begin{proposition}[Explicit form of normal-operator]
\label{prop:normal-factorization}
The map $\mathcal{M} $ extends to a bounded operator
\[
 \norm{\mathcal M}=\sqrt{\pi}:L^2(\D,\dd A)\longrightarrow\ell^2(\N_0^2),
\]
and
\begin{equation}
 K=\mathcal{M} ^*\mathcal{M},
 \qquad
 (Kf)(z)=\frac1\pi\int_{\D}
 \frac{f(w)}{|1-z\overline w|^2}\,\dd A(w).
 \label{eq:K-definition}
\end{equation}
In angular frequency $m$, the normal kernel is
\begin{equation}
 k_m(r,\rho)=\frac{(r\rho)^{|m|}}{1-r^2\rho^2}.
 \label{eq:angular-block-r}
\end{equation}
\end{proposition}

\begin{proof}
For a finitely supported sequence $c=(c_{p,q})$,
\[
 ( \mathcal{M}^*c)(z)
 =\frac1{\sqrt\pi}\sum_{p,q\geq0}c_{p,q}z^p\overline z^{\,q}.
\]
Consequently, finite partial sums give the kernel
\[
 \frac1\pi\sum_{p,q\geq0}z^p\overline z^{\,q}
 \overline w^{\,p}w^q
 =\frac1{\pi|1-z\overline w|^2}.
\]
The boundedness proved in Proposition~\ref{prop:norm} then extends the quadratic-form identity to all of $L^2$.  For fixed $m$, summing over $j$ gives
\[
 \sum_{j=0}^\infty(r\rho)^{|m|+2j}
 =\frac{(r\rho)^{|m|}}{1-r^2\rho^2}.
\]
\end{proof}

The Poisson-kernel identity reassembles the angular blocks into the
global kernel:
\begin{equation}
 \sum_{m\in\Z}\frac{(r\rho)^{|m|}}{1-r^2\rho^2}
 e^{im(\theta-\phi)}
 =\frac1{|1-z\overline w|^2}.
 \label{eq:block-sum}
\end{equation}
Under $t=r^2$ and $s=\rho^2$, the $m$th radial block is unitarily equivalent to
\begin{equation}
 (K_mh)(t)=\int_0^1\frac{(st)^{|m|/2}}{1-st}h(s)\,\dd s.
 \label{eq:block-t}
\end{equation}
The case $m=0$ is exactly \eqref{eq:radial-hilbert}.

\begin{remark}[Raw multiplicity]
With the same block convention as \eqref{eq:angular-block-r}, put $x=r\rho$.  The ordinary entrywise $\ell^2$ norm gives the raw block kernel
\[
 k_m^{\mathrm{raw}}(r,\rho)
 =(|m|+1)x^{|m|}+2\sum_{j=1}^\infty x^{|m|+2j}
 =\frac{2x^{|m|}}{1-x^2}+(|m|-1)x^{|m|}.
\]
Thus the raw block is twice the reduced block plus the rank-one edge term $(|m|-1)r^{|m|}\rho^{|m|}$.  For $m=0$ this is the familiar negative constant correction, for $|m|=1$ the correction vanishes, and for $|m|\geq2$ it is positive.  The normalization \eqref{eq:sharp-data-norm} removes exactly these exceptional edge terms and is the reason the blocks reassemble into the invariant kernel \eqref{eq:block-sum}.
\end{remark}

\section{Exact norm and non-compactness}
\label{sec:elementary}

The following simple arguments establish the operator is bounded and non-compact before any spectral diagonalization is used.

\begin{proposition}[Weighted Schur test]
\label{prop:norm}
The operator $K$ is bounded, positive and self-adjoint on $L^2(\D,\dd A)$, and
\begin{equation}
 \norm K=\pi.
 \label{eq:exact-norm}
\end{equation}
Consequently, $\norm{\mathcal{M}}=\sqrt\pi$.
\end{proposition}

\begin{proof}
For the non-negative symmetric kernel
\[
 \kappa(z,w)=\frac1{\pi|1-z\overline w|^2},
\]
we use the Schur test  in the form given, for example, in
\cite[Section~4]{HalmosSunder1978} with  weight $\chi(z)=(1-|z|^2)^{-1/2}$.  With $z=re^{i\theta}$ and $w=\rho e^{i\phi}$,
\[
 \int_0^{2\pi}\frac{\dd\phi}
 {1-2r\rho\cos\phi+r^2\rho^2}
 =\frac{2\pi}{1-r^2\rho^2}.
\]
Therefore
\begin{align}
 \int_{\D}\kappa(z,w)\chi(w)\,\dd A(w)
 &=2\int_0^1\frac{\rho\,\dd\rho}
 {(1-r^2\rho^2)\sqrt{1-\rho^2}}\notag\\
 &=\frac{2\arcsin r}{r\sqrt{1-r^2}}
 \leq \pi\chi(z).
 \label{eq:schur-calculation}
\end{align}
Schur's test gives $\norm K\leq\pi$.

The closed radial subspace reduces $K$.  The unitary map
\[
 (Wf)(t)=\sqrt\pi f(\sqrt t)
\]
identifies the radial restriction with $H$ in \eqref{eq:radial-hilbert}.  The standard near-extremizers $(1-t)^{-1/2+\epsilon}$ show $\norm H\geq\pi$, and hence $\norm K\geq\pi$.  Positivity follows from the factorization in Proposition~\ref{prop:normal-factorization}, and self-adjointness follows from symmetry and boundedness.
\end{proof}

\begin{proposition}[Boundary-annulus non-compactness]
\label{prop:noncompact}
The operator $K$ is not compact.
\end{proposition}

\begin{proof}
For $0<a<1$ set
\begin{equation}
 F_a(z)=\frac1{\sqrt{\pi a}}
 \mathbf1_{\{1-a<|z|^2<1\}}(z).
 \label{eq:annulus-sequence}
\end{equation}
Then $\norm{F_a}_2=1$ and $F_a\rightharpoonup0$ as $a\downarrow0$.  Under the radial unitary map, $F_a$ becomes
\[
 h_a(t)=a^{-1/2}\mathbf1_{(1-a,1)}(t).
\]
A direct calculation gives
\[
 (Hh_a)(t)=\frac1{\sqrt a\,t}
 \log\!\left(\frac{1-(1-a)t}{1-t}\right).
\]
Using $x=at/(1-t)$,
\begin{align*}
 \norm{Hh_a}_2^2
 &=\int_0^\infty\frac{\log^2(1+x)}{x^2}\,\dd x\\
 &=2\int_0^\infty\frac{\log(1+x)}{x(1+x)}\,\dd x
 =2\zeta(2)=\frac{\pi^2}{3}.
\end{align*}
Thus $\norm{KF_a}=\pi/\sqrt3$ for every $a$.  A compact operator maps a bounded weakly null sequence to a norm-null sequence, which is impossible here.
\end{proof}

\section{A global differential commutator}
\label{sec:commutator}

In \cite{Lionheart2026AtLast} we have a commuting differential operator for the radial case. In this section we extend that result to find a commuting partial differential operator for the general case. 

Set $t=r^2$.  The angular blocks \eqref{eq:block-t} suggest the differential expression
\begin{equation}
 \mathcal P
 =\partial_t\!\left(t(1-t)^2\partial_t\right)+t
 +\frac14\left(t+\frac1t\right)\partial_\theta^2.
 \label{eq:P-t}
\end{equation}
In $(r,\theta)$ variables,
\begin{equation}
 4\mathcal P
 =\frac1r\partial_r\!\left(r(1-r^2)^2\partial_r\right)
 +\left(r^2+\frac1{r^2}\right)\partial_\theta^2+4r^2.
 \label{eq:P-r}
\end{equation}
This expression is formally self-adjoint with respect to Euclidean area $r\dd r\dd\theta$.

\begin{proposition}[Direct kernel identity]
\label{prop:kernel-commutator}
For $z,w\in\D$,
\begin{equation}
 \mathcal P_z\kappa(z,w)=\mathcal P_w\kappa(z,w),
 \qquad
 \kappa(z,w)=\frac1{\pi|1-z\overline w|^2}.
 \label{eq:kernel-commutator}
\end{equation}
Consequently,
\begin{equation}
 \mathcal P Kf=K\mathcal P f,
 \qquad f\in C_c^\infty(\D).
 \label{eq:core-commutation}
\end{equation}
\end{proposition}

\begin{proof}
Let $\omega(z)=1-|z|^2$ and define
\begin{equation}
 \mathcal L=\mathcal P-\frac12\partial_\theta^2-1.
 \label{eq:L-def}
\end{equation}
A direct expansion gives
\begin{equation}
 \mathcal Lf=\frac{\omega}{4}\Delta_E(\omega f).
 \label{eq:L-euclidean}
\end{equation}
Introduce
\begin{equation}
 \delta(z,w)=\frac{\omega(z)\omega(w)}{|1-z\overline w|^2},
 \label{eq:delta}
\end{equation}
not to be confused with the Dirac $\delta$-function. For the Poincar\'e metric introduced in Section~\ref{sec:hyperbolic}, $\delta(z,w)=\sech^2(d_H(z,w)/2)$.  
For fixed $w$, the function $\delta(z,w)$ is radial about $w$ in the
hyperbolic metric:
\begin{equation}
\delta(z,w)=F(d_H(z,w)),
\qquad
F(d)=\sech^2(d/2).
\end{equation}
In geodesic polar coordinates centred at $w$,
\begin{equation}
\Delta_{H,z}F(d)=F''(d)+\coth(d)F'(d).
\end{equation}
Now
\begin{equation}
F'=-F\tanh(d/2),
\qquad
F''=F-\frac32F^2,
\end{equation}
and
\begin{equation}
\coth(d)F'
=
-F+\frac12F^2.
\end{equation}
Consequently,
\begin{equation}
F''+\coth(d)F'=-F^2.
\end{equation}
The same radial calculation, now in geodesic polar coordinates centred
at $z$, gives
\begin{equation}
\Delta_{H,z}\delta
=
-\delta^2
=
\Delta_{H,w}\delta.
\end{equation}
This equality is not a consequence merely of the symmetry
$\delta(z,w)=\delta(w,z)$; it follows because $\delta$ is the same
function of hyperbolic distance in either variable.

Since
\begin{equation}
\delta=\pi\omega(z)\omega(w)\kappa,
\end{equation}
we have
\begin{equation}\label{eq:L-kernel}
\begin{aligned}
	\mathcal L_z\kappa
	&=
	\frac{\omega(z)}4
	\Delta_{E,z}\bigl(\omega(z)\kappa\bigr)\\
	&=
	\frac{1}{\pi\omega(z)\omega(w)}
	\Delta_{H,z}\delta\\
	&=
	-\frac{\omega(z)\omega(w)}
	{\pi|1-z\overline w|^4}.
\end{aligned}
\end{equation}
The same computation in the $w$ variable gives the identical
expression, and hence
\begin{equation}
\mathcal L_z\kappa=\mathcal L_w\kappa.
\end{equation}


The kernel depends on the angular variables only through $\theta-\phi$, so $\partial_\theta^2\kappa=\partial_\phi^2\kappa$.  Equations \eqref{eq:L-def} and \eqref{eq:L-kernel} prove \eqref{eq:kernel-commutator}.  For compactly supported $f$, formal self-adjointness and integration by parts have no boundary contribution, giving \eqref{eq:core-commutation}.
\end{proof}

Equation \eqref{eq:core-commutation} alone does not specify a self-adjoint realization of the unbounded differential expression.  The canonical domain and full operator commutation will follow from the hyperbolic realization.

\section{Hyperbolic and Berezin realization}
\label{sec:hyperbolic}

Equip $\D$ with the Poincar\'e metric of curvature $-1$,
\begin{equation}
 \dd s_H^2=\frac{4|\dd z|^2}{(1-|z|^2)^2},
 \qquad
 \dd\mu_H=\frac{4\,\dd A}{(1-|z|^2)^2},
 \label{eq:poincare-metric}
\end{equation}
and Laplace--Beltrami operator \cite{bergeron2016spectrum}
\begin{equation}
 \Delta_H=\frac{(1-|z|^2)^2}{4}\Delta_E.
 \label{eq:hyperbolic-laplacian}
\end{equation}
Our sign convention makes $\Delta_H$ non-positive on $L^2(\D,\dd\mu_H)$.

Define
\begin{equation}
 U:L^2(\D,\dd A)\longrightarrow L^2(\D,\dd\mu_H),
 \qquad
 (Uf)(z)=\frac{\omega(z)}{2}f(z).
 \label{eq:U}
\end{equation}
This map is unitary.
\begin{theorem}[Simultaneous hyperbolic conjugation]
	\label{thm:hyperbolic-conjugation}
	Let $\mathcal L$ denote the differential expression
	\eqref{eq:L-euclidean}.  Then
	\begin{equation}
		U\mathcal L f=\Delta_HUf,
		\qquad f\in C_c^\infty(\D).
		\label{eq:L-conjugation-core}
	\end{equation}
	Consequently,
	\begin{equation}
		\mathsf L=U^{-1}\Delta_HU,
		\qquad
		\operatorname{Dom}(\mathsf L)
		=
		U^{-1}\operatorname{Dom}(\Delta_H),
		\label{eq:L-operator}
	\end{equation}
	defines a self-adjoint realization of $\mathcal L$ on
	$L^2(\D,\dd A)$.
	
	Moreover,
	\begin{equation}
		(UKU^{-1}g)(z)
		=
		\frac1{4\pi}\int_{\D}
		\sech^2\!\left(\frac{d_H(z,w)}2\right)
		g(w)\,\dd\mu_H(w).
		\label{eq:hyperbolic-convolution}
	\end{equation}
	Thus $U$ simultaneously conjugates $\mathsf L$ to the Poincar\'e
	Laplacian and $K$ to an isometry-invariant radial convolution on the
	hyperbolic plane.
\end{theorem}

\begin{proof}
	By \eqref{eq:L-euclidean},
	\[
	\mathcal Lf
	=
	\frac{\omega}{4}\Delta_E(\omega f).
	\]
	Since $Uf=\omega f/2$ and
	\[
	\Delta_H=\frac{\omega^2}{4}\Delta_E,
	\]
	we obtain, for $f\in C_c^\infty(\D)$,
	\[
	\begin{aligned}
		U\mathcal Lf
		&=
		\frac{\omega}{2}
		\left[
		\frac{\omega}{4}\Delta_E(\omega f)
		\right] \\
		&=
		\frac{\omega^2}{8}\Delta_E(\omega f)\\
		&=
		\frac{\omega^2}{4}
		\Delta_E\left(\frac{\omega f}{2}\right)\\
		&=
		\Delta_HUf.
	\end{aligned}
	\]
	This proves \eqref{eq:L-conjugation-core}.  Transporting the
	self-adjoint hyperbolic Laplacian through the unitary map $U$ gives
	the realization \eqref{eq:L-operator}.
	
	For the integral operator, $U^{-1}g=2g/\omega$, and
	\[
	\dd A(w)=\frac{\omega(w)^2}{4}\,\dd\mu_H(w).
	\]
	Consequently,
	\[
	\begin{aligned}
		(UKU^{-1}g)(z)
		&=
		\frac{\omega(z)}2
		\frac1\pi
		\int_{\D}
		\frac{2g(w)}
		{\omega(w)|1-z\overline w|^2}\,\dd A(w)\\
		&=
		\frac1{4\pi}
		\int_{\D}
		\frac{\omega(z)\omega(w)}
		{|1-z\overline w|^2}
		g(w)\,\dd\mu_H(w).
	\end{aligned}
	\]
	The hyperbolic distance identity
	\[
	\frac{\omega(z)\omega(w)}
	{|1-z\overline w|^2}
	=
	\sech^2\!\left(\frac{d_H(z,w)}2\right)
	\]
	then gives \eqref{eq:hyperbolic-convolution}.
\end{proof}

The invariant factor
\begin{equation}
 \delta(z,w)=\frac{(1-|z|^2)(1-|w|^2)}{|1-z\overline w|^2}
 \label{eq:hardy-overlap}
\end{equation}
is the squared overlap of normalized Hardy kernels
\[
 s_z(\zeta)=\frac{\sqrt{1-|z|^2}}{1-\overline z\zeta},
 \qquad
 |\ip{s_z}{s_w}_{H^2}|^2=\delta(z,w).
\]

\begin{remark}[The reduced data norm and the Hardy representation]
	The equal weighting of the non-redundant coordinates has a natural
	Hilbert-space interpretation. Let $\mathcal H^2(\D)$ denote the Hardy space of holomorphic functions
	\[
	f(z)=\sum_{n=0}^{\infty}a_nz^n,
	\qquad
	\norm{f}_{\mathcal H^2}^2=\sum_{n=0}^{\infty}|a_n|^2.
	\]
	We use the notation $\mathcal H^2$ here to distinguish the Hardy space
	from the Sobolev space $H^2(\D)$ used in Section~2.

	 With $e_p(\zeta)=\zeta^p$ denoting the
	standard orthonormal basis of $\mathcal{H}^2(\D)$, there are canonical
	identifications
	\[
	\ell^2(\N_0^2)
	\simeq
	\mathcal{H}^2(\D)\otimes\overline{\mathcal{H}^2(\D)}
	\simeq
	\mathcal S_2(\mathcal{H}^2(\D)),
	\]
	where $\mathcal S_2$ denotes the Hilbert--Schmidt operators.
	
	Under this identification, the doubly indexed vector
	\[
	\Psi_z
	=
	\bigl((1-|z|^2)\overline z^{\,p}z^q\bigr)_{p,q\geq0}
	\]
	corresponds to the rank-one projector (in the Bra-Ket notation)
	\[
	P_z=|s_z\rangle\langle s_z|.
	\]
	Consequently,
	\[
	\langle P_z,P_w\rangle_{\mathrm{HS}}
	=
	\operatorname{tr}(P_zP_w)
	=
	|\langle s_z,s_w\rangle_{\mathcal{H}^2}|^2
	=
	\delta(z,w).
	\]
	
	Thus the multiplicity-normalized $\ell^2$ norm is the natural
	Hilbert--Schmidt norm on the tensor-product data representation.
	The action of $g\in \operatorname{PSU}(1,1)$ carries $P_z$ to $P_{gz}$,
	so the corresponding normal operator belongs to the invariant
	convolution algebra on
	\[
	\D\simeq\operatorname{PSU}(1,1)/SO(2).
	\]
	This provides a representation-theoretic explanation for why the
	equal-weight reduction, rather than the raw matrix multiplicity,
	reveals the full hyperbolic symmetry.
\end{remark}


\subsection{Relation with Hardy and weighted Bergman Berezin transforms}\label{sec:berezin}

We briefly explain the Berezin terminology used below.  No result from
`Berezin quantization' is required for the diagonalization: the spectral
multiplier will be calculated directly in Section~\ref{sec:spectrum}. The purpose of this subsection is to show the relationship of our operator $K$ to a large body of literature. We recommend \cite{Stroethoff1997} for background on reproducing kernels, normalized kernel functions and the Berezin transform on classical analytic function spaces. For the specific connection to the invariant Laplacian see \cite{Englis1995}.

The reproducing kernel of $\mathcal H^2(\D)$ at $z$ is
\[
k_z(\zeta)=\frac1{1-\overline z\zeta},
\qquad
\norm{k_z}_{\mathcal H^2}^2
=
\frac1{1-|z|^2}.
\]
Its normalized form is therefore
\[
s_z(\zeta)
=
\frac{\sqrt{1-|z|^2}}{1-\overline z\zeta}.
\]
A direct calculation gives
\begin{equation}
	\left|
	\ip{s_z}{s_w}_{\mathcal H^2}
	\right|^2
	=
	\frac{(1-|z|^2)(1-|w|^2)}
	{|1-z\overline w|^2}
	=
	\delta(z,w).
	\label{eq:hardy-overlap1}
\end{equation}
Thus the invariant kernel in Theorem~\ref{thm:hyperbolic-conjugation}
is the squared modulus of the inner product between two normalized
Hardy reproducing kernels.

There is a standard one-parameter family containing this kernel.
For $s>1$, let $\mathcal A_s^2(\D)$ be the weighted Bergman space
with measure
\[
\dd\nu_s(w)
=
\frac{s-1}{\pi}(1-|w|^2)^{s-2}\,\dd A(w).
\]
The reproducing-kernel function at $z$ is
\[
k_z^{(s)}(\zeta)
=
\frac1{(1-\overline z\zeta)^s},
\]
and its normalized form is
\[
\widehat k_z^{(s)}(\zeta)
=
\frac{(1-|z|^2)^{s/2}}
{(1-\overline z\zeta)^s}.
\]
Hence
\[
\left|
\left\langle
\widehat k_z^{(s)},\widehat k_w^{(s)}
\right\rangle_{\mathcal A_s^2}
\right|^2
=
\delta(z,w)^s.
\]

The associated weighted Berezin transform is the integral operator
\begin{equation}
(B_sf)(z)
=
\frac{s-1}{4\pi}
\int_{\D}\delta(z,w)^s f(w)\,\dd\mu_H(w),
\qquad s>1.
	\label{eq:weighted-berezin}
\end{equation}
The coefficient $a-1$ is the normalization for which
$B_s 1=1$.

At $s=1$, the reproducing kernel
$(1-z\overline w)^{-1}$ is the Hardy, or Szeg\H{o}, kernel rather than
a weighted-area Bergman kernel.  The normalizing factor in
\eqref{eq:weighted-berezin} tends to zero, while the unnormalized
kernel has the finite limit $\delta(z,w)$.  Consequently, on compactly
supported functions,
\begin{equation}
	UKU^{-1}
	=
	\lim_{s\downarrow1}
	\frac{B_s}{s-1}.
	\label{eq:berezin-endpoint}
\end{equation}
For this reason we refer to $UKU^{-1}$ descriptively as the
renormalized Hardy-endpoint Berezin convolution.  We do not claim that
this terminology is universal.

The endpoint kernel
\[
 h(d)=\frac1{4\pi}\sech^2(d/2)
\]
is not in $L^1(\mathbb H^2)$: $2\pi h(d)\sinh d\to1$ as $d\to\infty$.  It is in $L^2$.  Thus boundedness should not be justified by an $L^1$ Young inequality.  Proposition~\ref{prop:norm} establishes boundedness independently, and the spherical transform below identifies the bounded multiplier.

The nonendpoint Berezin operators correspond, after conjugation to Euclidean area, to normal operators of weighted moment maps with coefficients \((s)_p(s)_q/(p!q!)\) and object weight \((1-|z|^2)^{s-1}\). Thus \(s=1\) is precisely the member selected by the equal-weight Calderón moments. All members have the same generalized eigenfunctions, and their multipliers differ only by gamma factors; in particular their high-frequency decay has the common exponential factor \(e^{-\pi\mu}\).

The use of a specially chosen Hilbert topology to expose a tractable
normal operator has precedents in the X-ray transform on the disk.  Mishra,
Monard and Zou study a one-parameter family of self-adjoint X-ray
normal operators obtained from singularly weighted $L^2$ structures.
These realizations admit generalized-Zernike singular systems and can
be expressed as functions of distinguished differential operators
\cite{MishraMonardZou2023}. 

\section{Hyperbolic spectral multiplier and operator spectrum}
\label{sec:spectrum}
By the Helgason--Fourier Plancherel theorem for the hyperbolic plane
\cite[Introduction, \S4]{Helgason1984}, the operator $-\Delta_H$ is
unitarily equivalent\footnote{Helgason uses the disk metric
	$
	ds_0^2=\frac{|\dd z|^2}{(1-|z|^2)^2},
	$
	which has curvature $-4$.  Our curvature $-1$ convention is
	$ds_H^2=4ds_0^2$, and hence
	$
	\Delta_H=\frac14\Delta_0$
	$d_H=2d_0.
	$
	Thus the spectral threshold $1$ in Helgason's normalization becomes
	$1/4$ in ours.  Equivalently, his spectral parameter $\lambda$ is
	related to ours by $\lambda=2\mu$.} to multiplication by
\[
\mu^2+\frac14,\qquad \mu\geq0.
\]
Consequently,
\[
\operatorname{spec}(-\Delta_H)=[1/4,\infty),
\]
and the spectrum is purely absolutely continuous.  For $\mu\geq0$, let
\[
\phi_\mu(d)
=
P_{-1/2+i\mu}(\cosh d),
\qquad
\Delta_H\phi_\mu
=
-\left(\mu^2+\frac14\right)\phi_\mu,
\qquad
\phi_\mu(0)=1.
\]
This is the normalized radial Laplace eigenfunction, traditionally
called the zonal spherical function.  The transform of a radial
kernel against $\phi_\mu$ is called the spherical transform; on
$\mathbb H^2$ it is the order-zero Mehler--Fock transform, or
equivalently the radial part of the Helgason--Fourier transform.
 Write
\begin{equation}
 \mathsf A=-\mathsf L-\frac14I\geq0.
 \label{eq:A}
\end{equation}

\begin{lemma}[Mehler--Fock transform of the endpoint kernel]
    \label{lem:mehler-fock}
	For $\mu\geq0$,
	\begin{equation}
		\int_1^\infty
		\frac{P_{-1/2+i\mu}(u)}{1+u}\,\dd u
		=
		\frac{\pi}{\cosh(\pi\mu)}.
		\label{eq:mehler-fock-integral}
	\end{equation}
	Consequently, the hyperbolic convolution kernel
	\[
	h(d)=\frac1{4\pi}\sech^2(d/2)
	\]
	has spectral multiplier
	\[
	\widehat h(\mu)=\frac{\pi}{\cosh(\pi\mu)}.
	\]
\end{lemma}

\begin{proof}
	Put
	\[
	\nu=-\frac12+i\mu,
	\qquad
	a=-\nu=\frac12-i\mu.
	\]
	Thus $\nu+1=1-a$.  By the hypergeometric representation of
	the Legendre function
	\cite[\href{https://dlmf.nist.gov/14.3.E6}{DLMF 14.3.6}]
	{NIST:DLMF},
	\[
	P_{-1/2+i\mu}(u)
	=
	{}_2F_1\left(
	a,1-a;1;\frac{1-u}{2}
	\right),
	\qquad u>1.
	\]
	
	Make the substitution
	\[
	t=\frac{u-1}{u+1},
	\qquad
	u=\frac{1+t}{1-t}.
	\]
	Then
	\[
	\frac{\dd u}{1+u}=\frac{\dd t}{1-t},
	\qquad
	\frac{1-u}{2}=-\frac{t}{1-t}.
	\]
	Pfaff's transformation
	\cite[\href{https://dlmf.nist.gov/15.8.E1}{DLMF 15.8.1}]
	{NIST:DLMF}
	gives
	\[
	{}_2F_1\left(
	a,1-a;1;-\frac{t}{1-t}
	\right)
	=
	(1-t)^a{}_2F_1(a,a;1;t).
	\]
	Therefore
	\[
	\int_1^\infty
	\frac{P_{-1/2+i\mu}(u)}{1+u}\,\dd u
	=
	\int_0^1
	(1-t)^{a-1}{}_2F_1(a,a;1;t)\,\dd t.
	\]
	
	Termwise integration of the hypergeometric series is justified by absolute convergence of the resulting beta series (its terms are $O(n^{-3/2})$), and yields
	\[
	\begin{aligned}
		\int_0^1
		(1-t)^{a-1}{}_2F_1(a,a;1;t)\,\dd t
		&=
		\frac1a\,{}_2F_1(a,a;a+1;1).
	\end{aligned}
	\]
	Gauss' summation formula
	\cite[NIST:DLMF 15.4.20]{NIST:DLMF}
	then gives
	\[
	\frac1a\,{}_2F_1(a,a;a+1;1)
	=
	\Gamma(a)\Gamma(1-a).
	\]
	Finally, Euler's reflection formula gives
	\[
	\Gamma(a)\Gamma(1-a)
	=
	\frac{\pi}{\sin(\pi a)}
	=
	\frac{\pi}{\cosh(\pi\mu)}.
	\]
	
	For the radial kernel $h$, geodesic polar coordinates give
	\[
	\begin{aligned}
		\widehat h(\mu)
		&=
		2\pi\int_0^\infty
		\frac1{4\pi}\sech^2(d/2)
		P_{-1/2+i\mu}(\cosh d)\sinh d\,\dd d \\
		&=
		\int_1^\infty
		\frac{P_{-1/2+i\mu}(u)}{1+u}\,\dd u,
	\end{aligned}
	\]
	where $u=\cosh d$.  This proves the multiplier formula.
\end{proof}

\begin{theorem}[Exact spectral formula]
\label{thm:spectral-formula}
On $L^2(\D,\dd A)$,
\begin{equation}
 K=\pi\sech\!\left(\pi\sqrt{-\mathsf L-\tfrac14I}\right).
 \label{eq:exact-spectral-formula}
\end{equation}
In particular,
\begin{equation}
 \spec(K)=[0,\pi],
 \label{eq:spectrum-K}
\end{equation}
$K$ has purely absolutely continuous spectrum, is injective, and has dense non-closed range.  Neither $0$ nor $\pi$ is an $L^2$ eigenvalue.
\end{theorem}

\begin{proof}
The radial convolution in Theorem~\ref{thm:hyperbolic-conjugation} is diagonalized by the spherical/Helgason Fourier transform.  Lemma~\ref{lem:mehler-fock} gives the bounded multiplier $\pi\sech(\pi\mu)$, yielding \eqref{eq:exact-spectral-formula}.  The multiplier is continuous, strictly positive for finite $\mu$, decreases from $\pi$ to zero, and the hyperbolic Laplacian has purely absolutely continuous spectrum.  The stated consequences follow from the spectral theorem.
\end{proof}

\begin{remark}
	The same calculation gives, for $\Re\alpha>1/2$,
	\[
	\int_1^\infty
	\frac{P_{-1/2+i\mu}(u)}{(1+u)^\alpha}\,\dd u
	=
	2^{1-\alpha}
	\frac{
		\Gamma(\alpha-\frac12+i\mu)
		\Gamma(\alpha-\frac12-i\mu)
	}{\Gamma(\alpha)^2}.
	\]
	The case $\alpha=1$ is the endpoint kernel used above.  This
	general identity also displays directly the gamma-function multipliers
	of the weighted Berezin family.
\end{remark}

The formula also supplies the canonical closed realization of the commutator.  Let
\begin{equation}
 J=-i\partial_\theta.
 \label{eq:J}
\end{equation}
Rotations are hyperbolic isometries, so $\mathsf A$ and $J$ strongly commute.  Define
\begin{equation}
 \mathsf P=\frac34I-\mathsf A-\frac12J^2,
 \qquad
 \Dom(\mathsf P)=\Dom(\mathsf A)\cap\Dom(J^2).
 \label{eq:P-operator}
\end{equation}

\begin{corollary}[Strong commutation]
\label{cor:strong-commutation}
The operator $\mathsf P$ is self-adjoint, agrees with \eqref{eq:P-t} on $C_c^\infty(\D)$, and
\begin{equation}
 K\Dom(\mathsf P)\subseteq\Dom(\mathsf P),
 \qquad
 \mathsf P Kf=K\mathsf P f
 \quad(f\in\Dom(\mathsf P)).
 \label{eq:strong-commutation}
\end{equation}
Indeed, $K$ strongly commutes with the spectral projections of $\mathsf P$.
\end{corollary}

\begin{proof}
In the joint spectral representation of $(\mathsf A,J)$, the operators $K$ and $\mathsf P$ are multiplication by
\[
 \frac{\pi}{\cosh(\pi\mu)}
 \quad\hbox{and}\quad
 \frac34-\mu^2-\frac{m^2}{2},
\]
respectively.  The first multiplier is bounded, so it preserves the domain defined by square integrability against the second multiplier and commutes pointwise with it.
\end{proof}

\section{Generalized singular functions and resolution}
\label{sec:singular-functions}

Let $x=d_H(0,r)=2\operatorname{arctanh}r$, so $r=\tanh(x/2)$.  Separating angular frequency $m\in\Z$, a regular generalized eigenfunction of the hyperbolic Laplacian is
\begin{align}
 R_{m,\mu}(x)
 &=\tanh(x/2)^{|m|}
 {}_2F_1\!\left(\frac12+i\mu,\frac12-i\mu;
 |m|+1;-\sinh^2(x/2)\right)\notag\\
 &=\Gamma(|m|+1)P_{-1/2+i\mu}^{-|m|}(\cosh x).
 \label{eq:radial-genfun}
\end{align}
The corresponding Euclidean-space generalized function is
\begin{equation}
 \Phi_{m,\mu}(r,\theta)
 =\frac{2}{1-r^2}R_{m,\mu}(x)e^{im\theta}.
 \label{eq:euclidean-genfun}
\end{equation}
Formally,
\begin{equation}
 K\Phi_{m,\mu}
 =\frac{\pi}{\cosh(\pi\mu)}\Phi_{m,\mu}.
 \label{eq:gen-eigen-K}
\end{equation}
These are generalized, not square-integrable, eigenfunctions.

Because $K=\mathcal{M}^*\mathcal{M}$, the generalized singular attenuation of the measurement map is
\begin{equation}
 \sigma(\mu)=\left(\frac{\pi}{\cosh(\pi\mu)}\right)^{1/2}
 \sim\sqrt{2\pi}\,e^{-\pi\mu/2}
 \qquad(\mu\to\infty).
 \label{eq:singular-attenuation}
\end{equation}
Thus an $L^2$ perturbation component at hyperbolic frequency $\mu$ is exponentially suppressed in the data.  In a deterministic noise model $d^\delta=\mathcal{M} f+e$, $\norm e\leq\delta$, inversion of that component amplifies noise by approximately $\sigma(\mu)^{-1}$.  This is the appropriate continuous-spectrum analogue of exponentially decaying singular values.

The stable Liouville profile
\begin{equation}
 v_{m,\mu}(x)=\sqrt{\sinh x}\,R_{m,\mu}(x)
 \label{eq:liouville-profile}
\end{equation}
satisfies
\begin{equation}
 -v''+\left(m^2-\frac14\right)\csch^2x\,v=\mu^2v.
 \label{eq:poschl-teller}
\end{equation}
The potential is short range as $x\to\infty$, so $v$ is asymptotically oscillatory with wavelength $\pi/\mu$ between successive zeros.  In Euclidean radius, equal increments of $x$ are compressed exponentially towards $r=1$.  This gives a direct geometric explanation of the concentration of fine resolution near the boundary.

The exponential multiplier should not be confused with compactness.  It is a nonzero multiplication function on a continuous spectral space; multiplication by such a function is not compact.  Proposition~\ref{prop:noncompact} gives an elementary manifestation of the same fact.

\section{Numerical study with truncated radius}
\label{sec:numerics}

For illustration, we truncate the hyperbolic radial variable to
$0<x<X$.  For each angular frequency $|m|\leq M$, a Gauss--Legendre
quadrature rule is used to form a symmetric Nystr\"om matrix for the
corresponding truncated integral operator.  If $x_i$ and $w_i$ are
the quadrature nodes and weights and $\widetilde k_m$ denotes the
kernel in the transformed coordinate, the matrix has the form
\[
(A_m)_{ij}
=
\sqrt{w_i}\,\widetilde k_m(x_i,x_j)\sqrt{w_j}.
\]
The leading eigenpairs of the real symmetric matrix $A_m$ are computed
using MATLAB's \texttt{eigs} routine.  If
$\lambda_{m,j}^{(X,N)}$ is a computed eigenvalue, then
\[
\sigma_{m,j}^{(X,N)}
=
\sqrt{\lambda_{m,j}^{(X,N)}}
\]
is the corresponding singular value of the truncated measurement
map.

The computed eigenvalues and eigenvectors are compared with the
continuous spectral formula in two independent ways.  First, an
effective spectral parameter is obtained from
\[
\mu_\lambda
=
\frac1\pi
\operatorname{arcosh}\!\left(\frac{\pi}{\lambda_{m,j}^{(X,N)}}\right),
\]
which inverts the analytic multiplier
$\lambda(\mu)=\pi/\cosh(\pi\mu)$.  Second, when sufficiently many
radial zeros are present, $\mu$ is estimated from their asymptotic
spacing $\Delta x\sim\pi/\mu$.  After the Liouville rescaling
\eqref{eq:liouville-profile}, the numerical eigenvectors are compared
directly with the associated Legendre/hypergeometric generalized
eigenfunctions.

The truncation makes every block compact; its discrete modes should be interpreted as spectral packets sampling the continuous parameter $\mu$, not as convergent eigenvalues of the infinite operator.

Figure~\ref{fig:global-order} shows the characteristic interlacing of angular frequency and radial oscillation count.  The sawtooth behaviour is expected: after ordering by attenuation, the next mode need not have either the next angular frequency or the next radial zero count, as observed in \cite{adler2021}.  Figure~\ref{fig:analytic-comparison} compares selected numerical profiles with \eqref{eq:radial-genfun} after the stable transformation \eqref{eq:liouville-profile}.

\begin{figure}[htbp]
 \centering
 \includegraphics[width=1.0\textwidth]{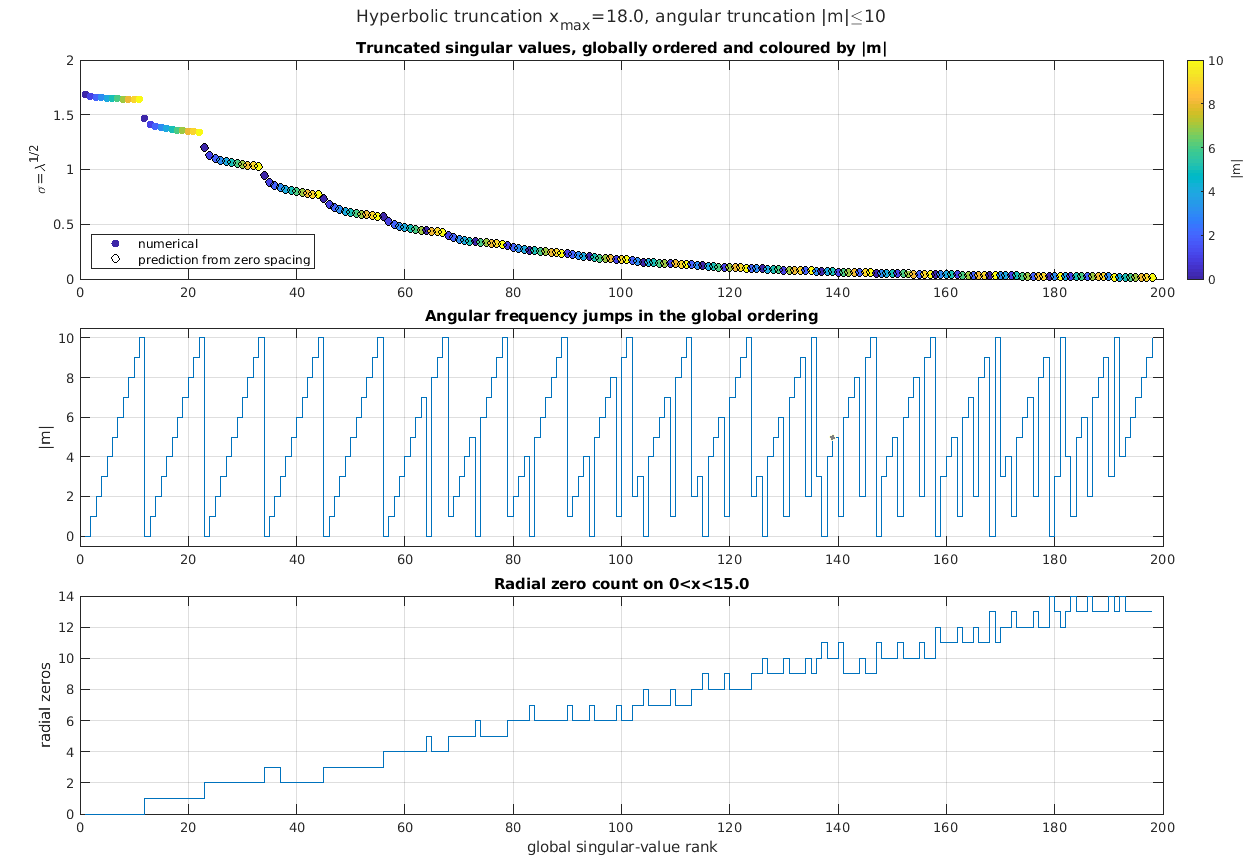}
 \caption{Finite-section global ordering.  The computed singular values $\sigma=\sqrt\lambda$ are compared with the square root of the analytic normal-operator multiplier, while the lower panels display the angular frequency and radial-zero count of the correspondingly ordered mode.  The discrete ordering depends on the hyperbolic truncation, but the mixed angular--radial pattern is robust.}
 \label{fig:global-order}
\end{figure}

\begin{figure}[htbp]
 \centering
 \includegraphics[width=1.0\textwidth]{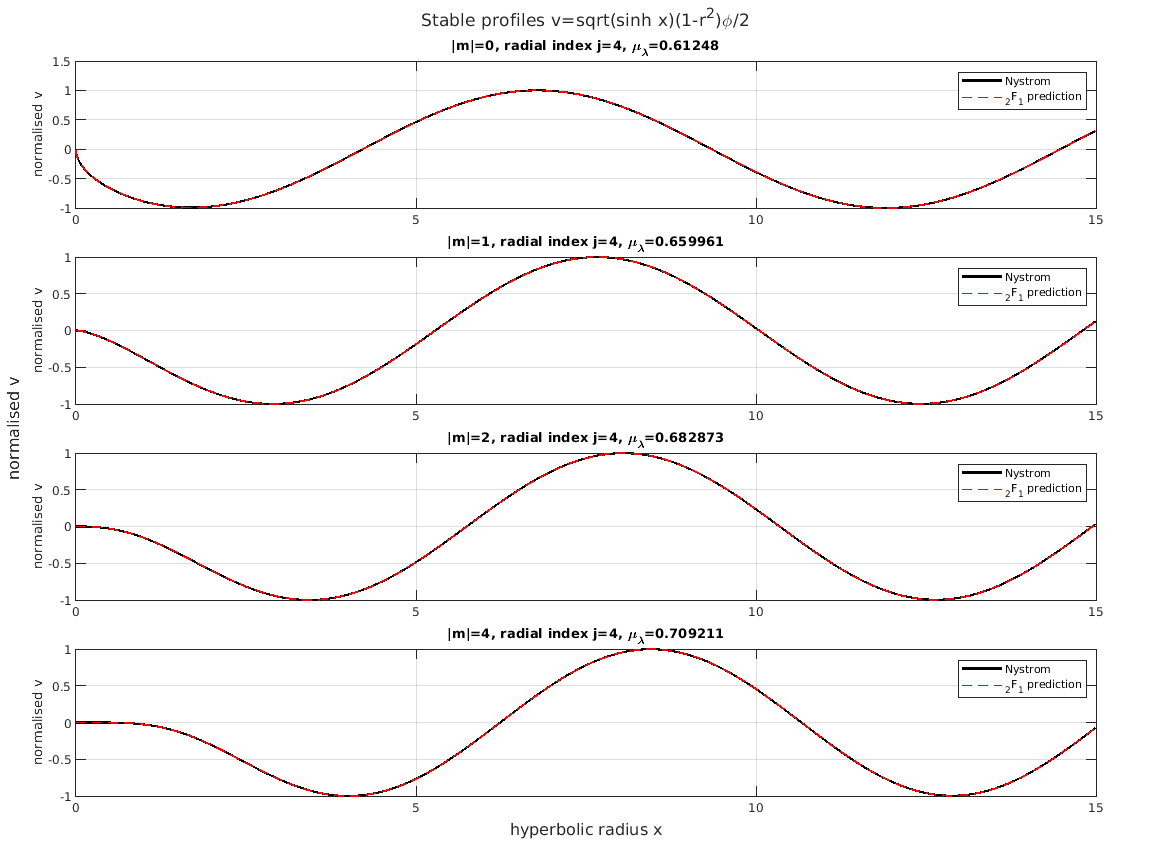}
 \caption{Comparison of finite-section numerical modes and associated Legendre/hypergeometric profiles in the Liouville variable.  The agreement away from the artificial boundary is a stronger check than singular-value matching alone.}
 \label{fig:analytic-comparison}
\end{figure}
\FloatBarrier

Although the spectral decomposition is continuous, finite-dimensional
discretizations produce ordinary singular vectors.  These should be
interpreted as spectral packets associated with the generalized
singular functions $\Phi_{m,\mu}$, rather than as approximations to
isolated eigenfunctions of the infinite normal operator.

Figure~\ref{fig:nonradialsingularfunctions} displays selected finite-section modes
ordered by their computed attenuation.  The ordering interlaces angular
frequency and radial oscillation count: increasing the global mode
number need not increase either quantity separately.  This behaviour,
seen in earlier numerical singular value decompositions of EIT
operators, follows naturally from the joint spectral parameters
$(m,\mu)$.

In the hyperbolic radial coordinate
\[
x=2\operatorname{arctanh}r,
\]
the transformed radial functions are asymptotically oscillatory with
approximately constant wavelength $\pi/\mu$.  The compression of equal
$x$-intervals as $r\to1$ explains the apparent concentration of radial
resolution near the Euclidean boundary.

\begin{figure}
	\centering
	\includegraphics[width=1.0\linewidth]{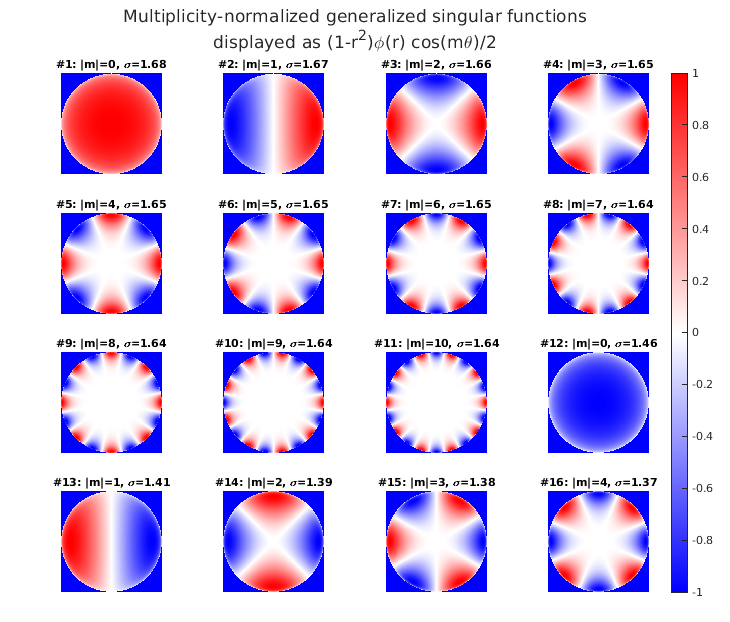}
	\caption{
        Selected finite-section approximations to the generalized right
        singular functions of the multiplicity-normalized linearized
        measurement map.  Each panel is labelled by angular frequency
        $|m|$, finite-section radial index $j$, radial zero count $z$, and
        computed singular value $\sigma$.  The modes are ordered globally by
        the computed normal-operator eigenvalue.  The resulting sequence
        interlaces angular and radial complexity.  The increasing concentration
        of radial oscillations near $r=1$ is the Euclidean manifestation of
        approximately uniform oscillation in hyperbolic radius.  For $|m|>0$,
        a second real mode is obtained by replacing the cosine angular factor
        by a sine factor.
}
	\label{fig:nonradialsingularfunctions}
\end{figure}

\section{Discussion}
\label{sec:discussion}

The result supplies a global counterpart of the rotationally symmetric calculation.  The one-dimensional Hilbert-kernel operator is the $m=0$ angular block of a Hardy-endpoint Berezin convolution on the Poincar\'e disk.  The separate Sturm--Liouville expressions are the angular reductions of one disk differential operator.  This explains why the same multiplier appears through finite Hilbert transforms, Hilbert matrices, scattering theory and hyperbolic harmonic analysis.

The relation with disk Radon transforms is structural rather than literal.  In Radon and geodesic X-ray problems, Zernike or generalized Zernike functions diagonalize distinguished disk operators, and differential intertwining identities yield sharp mapping results \cite{Monard2020,MishraMonard2021,MishraMonardZou2023,EptaminitakisMonardZou2026}.  Here the same Zernike lattice arises from the multiplicities of boundary Fourier measurements, but the normal operator is not an X-ray normal operator and is not localized the unknown along hyperbolic geodesics, even though Berenstein and 
Casadio Tarabusi \cite{BerensteinCasadio1994} were able to connect this transform with early Barber-Brown backprojection.   The geometry is hyperbolic while the imaging mechanism is an exponentially smoothing invariant convolution.

The abstract diagonalization belongs to established Berezin and symmetric-space harmonic analysis.  The new inverse-problem statement is the route
\begin{equation*}
\begin{gathered}
 \hbox{linearized Calder\'on matrix}
 \longrightarrow \hbox{multiplicity fibres}
 \longrightarrow \hbox{Zernike lattice}\\
 \longrightarrow \frac1{\pi|1-z\overline w|^2}
 \longrightarrow \hbox{Hardy-endpoint Berezin convolution}.
\end{gathered}
\end{equation*}
This route makes the data norm and the exceptional edge $n=|m|$ visible.

The exact diagonalization is therefore a statement about a specified local experiment: the derivative at the zero potential, with one equal-weight coordinate for each independent Fourier moment.  
While this is not the a norm derived from a specific physical noise model. However it exposes a global invariant geometry that is hidden by the unreduced matrix norm, and it leaves the exponential attenuation already present in the uniformly equivalent radial subproblem.  These points allow other data norms---including covariance-weighted electrode norms---to be compared objectively with the present baseline.

Implicit in our decomposition is a possible route to a regularized inversion formula, however it is not clear that this would have a practical value over matrix based regularized inversion methods. The D-bar method gives an explicit non-linear inversion method for the Calder\'on problem in two dimensions. In that case the  synthesis of complex geometric optics (CGO) data from boundary current and voltage data is severely illposed and requires regularization in practice.  See \cite{hamilton2019comparing} for a comparison between D-bar and standard regularized inversion. A diagonalization of the normal operator of the voltage and current to CGO data would provide an interesting comparison with the present work. 

The extension to the ball in dimension three, the Neumann-to-Dirichlet case and  further understanding of the effect of unknown conductivity near the boundary remain outstanding and interesting questions for future work.  This paper however makes an important contribution in confirming the spectral decay and depth dependent resolution long observed numerically.  It also raises interesting issues more generally for diagonalization of self adjoint integral operator: specifically systematic methods to discover commuting differential operators.


\ack{This work was developed with substantial assistance from ChatGPT (OpenAI), using GPT-5.6 Sol Pro through a ChatGPT Pro subscription in September 2026; OpenAI provided access free of charge.  The assistance included proposing candidate derivations, symbolic and numerical code, literature-search suggestions, and draft prose.  In particular, the system proposed the hyperbolic/Berezin interpretation.  The author formulated the inverse problem and the multiplicity normalization, derived the normal kernel, selected and developed the final arguments, executed and reviewed the computations, checked the cited identities, and assumes responsibility for all mathematical claims, references and originality statements.
	
The author would like to thank Andy Adler, David Isaacson, Fancois Monard, Mikko Salo and Alexander Katsevich for helpful suggestions improving  earlier drafts.

 The author would like to thank the Isaac Newton Institute for Mathematical Sciences, Cambridge, for support and hospitality during the programme Rich and Non-linear Tomography where some of the work on this paper was undertaken. This work was supported by EPSRC grant no EP/Z000580/1.  Part of this work was completed while the author was a Visiting Fellow at Clare Hall, Cambridge.}

\roles{William R B Lionheart: conceptualization, methodology, formal analysis, investigation, software, validation, visualization, writing--original draft, and writing--review and editing.}
\data{The MATLAB code used for the exploratory finite-section figures can be found at the repository \url{https://github.com/billlion/CalderonDiskSVD} }

\bibliographystyle{iopart-num-titles}
\bibliography{references}

\appendix
\section{A direct polar-coordinate check of the kernel identity}
\label{app:polar-check}

Set
\[
 \psi=\theta-\phi,\qquad
 D=1-2r\rho\cos\psi+r^2\rho^2,
 \qquad \kappa=(\pi D)^{-1}.
\]
The operator acting in the $z$ variable is
\[
 \mathcal P_z=
 \frac1{4r}\partial_r\bigl(r(1-r^2)^2\partial_r\bigr)
 +\frac14(r^2+r^{-2})\partial_\psi^2+r^2,
\]
and $\mathcal P_w$ is obtained by replacing $r$ by $\rho$.  Expansion over the common denominator $\pi D^3$ gives a numerator divisible by $r^2-\rho^2$ from the radial terms; the angular and zeroth-order terms cancel the remainder, leaving
\[
 (\mathcal P_z-\mathcal P_w)\kappa=0.
\]
This can be checked in Mathematica as:
\begin{lstlisting}[language=Mathematica,basicstyle=\ttfamily\footnotesize,frame=single,breaklines=true]
ClearAll[r, rho, psi, ker, Pz, Pw];
den = 1 - 2 r rho Cos[psi] + r^2 rho^2;
ker = 1/(Pi den);
Pz[u_] := 1/(4 r) D[r (1-r^2)^2 D[u,r],r]
          + 1/4 (r^2 + 1/r^2) D[u,{psi,2}] + r^2 u;
Pw[u_] := 1/(4 rho) D[rho (1-rho^2)^2 D[u,rho],rho]
          + 1/4 (rho^2 + 1/rho^2) D[u,{psi,2}] + rho^2 u;
FullSimplify[Pz[ker]-Pw[ker],
 Assumptions -> 0<r<1 && 0<rho<1 && Element[psi,Reals]]
\end{lstlisting}
which evaluates to zero.

\section{Reality of generalized eigenfunctions}

For the integral operator
\[
(Kf)(t)=\int_0^1\frac{f(s)}{1-ts}\,ds,
\]
the generalized eigenfunctions may be written in the form
\[
\Phi_\mu(t)
=
(1-t)^{-1/2+i\mu}
{}_2F_1\left(
\frac12+i\mu,\frac12+i\mu;
1;t
\right),
\qquad
0<t<1,
\]
where \(\mu\in\mathbb{R}\). Although both the complex power and the hypergeometric function appearing in this expression are in general complex, their phases cancel and \(\Phi_\mu\) is real-valued.

To see this, put
\[
F_\mu(t)
=
{}_2F_1\left(
\frac12+i\mu,\frac12+i\mu;
1;t
\right).
\]
Euler's transformation
\[
{}_2F_1(a,b;c;t)
=
(1-t)^{c-a-b}
{}_2F_1(c-a,c-b;c;t)
\]
gives, with
\[
a=b=\frac12+i\mu,
\qquad c=1,
\]
the identity
\[
F_\mu(t)
=
(1-t)^{-2i\mu}
{}_2F_1\left(
\frac12-i\mu,\frac12-i\mu;
1;t
\right).
\]
For real \(t\in(0,1)\) and real \(\mu\),
\[
{}_2F_1\left(
\frac12-i\mu,\frac12-i\mu;
1;t
\right)
=
\overline{F_\mu(t)}.
\]
Consequently,
\[
F_\mu(t)
=
(1-t)^{-2i\mu}\overline{F_\mu(t)}.
\]
Multiplication by (\(1-t)^{i\mu}\) gives
\[
(1-t)^{i\mu}F_\mu(t)
=
(1-t)^{-i\mu}\overline{F_\mu(t)}
\overline{(1-t)^{i\mu}F_\mu(t)}.
\]
It follows that
\[
(1-t)^{i\mu}
{}_2F_1\left(
\frac12+i\mu,\frac12+i\mu;
1;t
\right)
\]
is real-valued. Hence
\[
\Phi_\mu(t)\in\mathbb{R},
\qquad 0<t<1,\quad \mu\in\mathbb{R}.
\]

Thus the apparently complex representation
\[
\Phi_\mu(t)
=(1-t)^{-1/2}
\left[
(1-t)^{i\mu}
{}_2F_1\left(
\frac12+i\mu,\frac12+i\mu;
1;t
\right)
\right]
\]
is in fact a representation of a real generalized eigenfunction.

This also clarifies the relation with the Friedrichs realization of the commuting differential operator. At \(t=0\), the corresponding differential equation has a non-logarithmic Frobenius solution and a second, logarithmic solution. Since
\[
{}_2F_1(a,b;1;0)=1,
\]
we have
\[
\Phi\mu(0)=1,
\]
and \(\Phi_\mu\) is the non-logarithmic solution normalized at (t=0). It therefore gives the generalized eigenfunction selected by the Friedrichs boundary condition, up to an overall real normalization.

Under the Liouville transformation
\[
t=\tanh^2\frac{\xi}{2},
\]
the corresponding Schrödinger generalized eigenfunction is
\[
u_\mu(\xi)
=
\bigl[t(1-t)^2\bigr]^{1/4}\Phi_\mu(t).
\]
Equivalently, writing
\[
r=\tanh\frac{\xi}{2},
\qquad t=r^2,
\]
we obtain
\[
u_\mu(\xi)
=	
\sqrt{r(1-r^2)},
\Phi_\mu(r^2).
\]
This function is again real. As \(\xi\to\infty\), it has the real standing-wave form
\[
u_\mu(\xi)
\sim
A(\mu)\cos\bigl(\mu\xi+\delta(\mu)\bigr),
\]
or, equivalently,
\[
u_\mu(\xi)
\sim
C(\mu)e^{i\mu\xi}
+
\overline{C(\mu)}e^{-i\mu\xi}.
\]
Thus the complex factors in the hypergeometric representation should not be interpreted as selecting a single travelling wave. The Friedrichs generalized eigenfunction is real, and its asymptotic behaviour consists of the corresponding conjugate pair of incoming and outgoing waves.


An alternative way to see that \(\Phi_\mu\) is real is obtained from Pfaff's transformation
\[
{}_2F_1(a,b;c;t)
=
(1-t)^{-a}
{}_2F_1\left(
a,c-b;c;\frac{t}{t-1}
\right).
\]
Taking
\[
a=b=\frac12+i\mu,\qquad c=1,
\]
gives
\[
{}_2F_1\left(
\frac12+i\mu,\frac12+i\mu;1;t
\right)
=
(1-t)^{-1/2-i\mu}
{}_2F_1\left(
\frac12+i\mu,\frac12-i\mu;
1;\frac{t}{t-1}
\right).
\]
Consequently,
\[
\Phi_\mu(t)
=
\frac{1}{1-t}
{}_2F_1\left(
\frac12+i\mu,\frac12-i\mu;
1;\frac{t}{t-1}
\right).
\]
The two upper parameters are now complex conjugates, and the argument is real and negative for (0<t<1). In particular, this representation makes the fact that  \(\Phi_\mu\) is real considerably more transparent.

Using
\[
P_\nu(x)
=
{}_2F_1\left(
-\nu,\nu+1;1;\frac{1-x}{2}
\right),
\]
where \(P_\nu\) is the Legendre function of the first kind, we may equivalently write
\[
\Phi_\mu(t)
=	
\frac{1}{1-t}
P_{-1/2+i\mu}\left(\frac{1+t}{1-t}\right).
\]
Here \(P_{-1/2+i\mu}(x)\), for \(x>1\),  is real-valued for real \(\mu\).  (see DLMF 14.20 Conical (or Mehler) Functions \url{https://dlmf.nist.gov/14.20})

Finally, under the Liouville change of variable
\[
t=\tanh^2\frac{\xi}{2},
\]
we have
\[
\frac{1+t}{1-t}=\cosh\xi,
\qquad
\frac{1}{1-t}=\cosh^2\frac{\xi}{2},
\]
and hence
\[
\Phi_\mu\left(\tanh^2\frac{\xi}{2}\right)
=	
\cosh^2\frac{\xi}{2},
P_{-1/2+i\mu}(\cosh\xi).
\]
Thus the cancellation of the complex phases in the original hypergeometric representation can also be understood through its representation as a real conical function.

%

\end{document}